\documentclass[letterpaper,10pt,conference,twoside]{ieeeconf}

\IEEEoverridecommandlockouts                              % This command is only needed if 
\usepackage{amsmath,amssymb,color,cite}%,showkeys}

\usepackage{graphicx}
\usepackage{dblfloatfix}
\usepackage{latexsym}
\usepackage{algorithm}
\usepackage{algpseudocode}
\usepackage{verbatim}
\usepackage{cite}
\usepackage[latin1]{inputenc}
\usepackage{enumerate}

\usepackage{amsthm}
\usepackage{bbm}

\usepackage[font=footnotesize]{caption}
\usepackage{subcaption}
\usepackage{arydshln}
\usepackage{faktor}
\usepackage{amsfonts}
\usepackage{amsbsy}
\usepackage{bm}
\usepackage{booktabs}
\usepackage{aligned-overset}
\usepackage{extarrows}

\DeclareMathOperator{\rank}{rank}

\newcommand{\norm}[1]{\ensuremath{\| #1 \|}}

\let\leq\leqslant
\let\geq\geqslant
\let\emptyset\varnothing

\newcommand{\Z}{\mathbb{Z}}

\newcommand{\R}{\mathbb{R}}
\newcommand{\C}{\mathbb{C}}

\newcommand{\B}{\mathcal{B}}
\newcommand{\T}{\mathcal{R}}
\newcommand{\M}{\mathcal{M}}
\newcommand{\A}{\mathcal{A}}

\theoremstyle{definition}
\newtheorem{theorem}{Theorem}[section]

\newtheorem{lemma}[theorem]{Lemma}
\newtheorem{corollary}[theorem]{Corollary}

\newtheorem{remark}[theorem]{Remark}

\newtheorem{example}[theorem]{Example}

\DeclareMathOperator{\dif}{dif} 
\renewcommand{\bar}{\overline}

\title{Stability, Contraction, and Controllers for Affine Systems}
\author{
	L. P. Wieringa\thanks{L. P. Wieringa is with the Department of Mathematics, ETH Z\"urich, Switzerland, {\tt\footnotesize lwieringa@student.ethz.ch}. A.~Padoan is with the Department of Electrical and Computer Engineering, University of British Columbia, Canada, {\tt\footnotesize alberto.padoan@ubc.ca}. F. D\"orfler is with the Department of Information Technology and Electrical Engineering, ETH Z\"urich, Switzerland, {\tt\footnotesize doerfler@control.ee.ethz.ch}. J.~Eising is with ENTEG, University of Groningen, the Netherlands, {\tt\footnotesize j.eising@rug.nl}. This work is supported by the SNF/FW Weave Project 200021E\_20397 and the Natural Sciences and Engineering Research Council of Canada (NSERC), grant numbers: RGPIN-2025-06895 and DGECR-2025-00382.},
	A.~Padoan, %
	F. D\"{o}rfler, 
	J. Eising
}
	
\begin{document}

\maketitle

\begin{abstract}
	Recent developments in data-driven control have revived interest in the behavioral approach to systems theory, where systems are defined as sets of trajectories rather than being described by a specific model or representation. However, most available results remain confined to linear systems, limiting the applicability of recent methods to complex behaviors. Affine systems form a natural intermediate class: they arise from linearization, capture essential nonlinear effects, and retain sufficient structure for analysis and design. This paper derives necessary and sufficient conditions independent of any particular representation for three fundamental stability problems for affine behaviors: (i) converse Lyapunov theorems for contraction of input-output systems; (ii) implementability and existence of prescribed contractive references; and (iii) whether these references can be implemented with linear or affine controllers. For the latter, we show that linear controllers suffice for implementing a contractive closed-loop, and and affine controllers are needed for equilibrium placement. 
\end{abstract}

\section{Introduction}
In addition to arising in physical modeling, affine models have recently been used extensively in nonlinear data-based predictive control. Motivated by the use of linearizations of time-varying (e.g. \cite{Berberich_2022,Berberich2024}) or nonlinear systems (e.g. \cite{Naef2025,Beerwerth2025}), these methods employ \textit{data-based representations}, where the set of finite-length trajectories is characterized as all affine combinations of a set of previously collected measurements. Since state measurements are often not available, the use of state-space models is inhibited. To provide convergence guarantees for such predictive control schemes, typical requirements are controllability or stabilizability of the system \cite{rawlings2017model}. However, analysis of these properties for affine models typically relies on the realization of state-space representations, and is not derived from first principles but treated as a corollary of the case of linear systems.

This paper stands in the context of the behavioral approach (see e.g. \cite{Polderman1998}), which considers systems irrespective of their representation. It was recently shown that affine systems exhibit minor but fundamental differences with respect to linear systems. Particularly relevant here are the recent investigations of the fundamental lemma (cf. \cite{willems2005}) in the affine setting \cite{martinelli2022data,Berberich_2022,Padoan2023}. Moreover, the paper  \cite{Padoan2025} detailed realization theory and controllability of such systems. In line with these, this paper focuses on stability theory.

For nonlinear systems in the behavioral setting, we cannot check for stability in a tractable manner. In the case of linear systems with a state-space representation, stability analysis is often done by calculating a quadratic Lyapunov function of the state. If states are not available, the behavioral approach uses \textit{quadratic difference forms (QDFs)} as Lyapunov functions, for which the equivalence between stability and the existence of such a Lyapunov function was proven in continuous-time \cite{Willems1998}, and later extended to discrete-time \cite{Kojima2006}. In the linear case, QDFs have found use in analysis and controller synthesis in the direct data-based settings \cite{Maupong2017, vanWaarde2022, vanWaarde2026}. Since quadratic Lyapunov functions of the state have many uses in nonlinear analysis, QDFs are the natural framework for the analysis of affine systems. 

Once conditions for the analysis are in place, we are interested in the design of controllers for non-autonomous systems. In the behavioral setting, this problem is approached in terms of two related problems. On the one hand, the question of whether a given reference behavior can be \textit{implemented}, and on the other, whether a stable and implementable reference exists. The first of these problems has been studied in full generality \cite{vanderSchaft2003}, but there are no known necessary and sufficient conditions for implementability tailored to specific model classes beyond the linear case. Similarly, the question of whether there exists a stable and implementable reference has only been studied in the linear case \cite{Willems1997, Belur2001, Belur2002}.

\textit{Contributions.} In this paper, we investigate stability for affine systems. Unlike linear systems, affine systems need not contain the origin, making classical stability ill-defined; contraction  (see e.g. \cite{Bullo2024}, \cite{WL-JJES:98}) sidesteps this issue by comparing pairs of trajectories. We prove a number of fundamental results regarding stability and contraction:
\begin{enumerate}
	\item We prove, akin to converse Lyapunov results, that an affine system is contractive if and only if there exists a quadratic difference form as a contraction metric.
	\item We derive necessary and sufficient conditions for a given reference behavior to be implementable using a controller, and for the existence of an affine, implementable, and contractive reference behavior.
	\item We prove the fact that for the implementation of a contractive reference for an affine system, it is sufficient to consider linear controllers, and that using affine controllers allows for the placement of the resulting equilibrium.
\end{enumerate} 
In order to be equally suited to physics-derived affine models and data-based representations, we will work representation free, in the behavioral setting.

\textit{Notation.} We denote by $\Z_+, \Z$, $\R$ and $\C$ the set of nonnegative integers, integers, reals and complex numbers respectively. The set of all maps from $\Z$ to $\R^q$ is denoted $(\R^q)^\Z$. Given $\tau\in \Z$, we denote by $\sigma^\tau$ the map that sends $w\in (\R^q)^\Z$ to $\sigma^\tau w\in (\R^q)^\Z$ given by $\sigma^\tau w(t) = w(t+\tau)$. If $\tau=1$, we omit the superscript, and if $\sigma^\tau w =w$ for all $\tau\in\Z$, we say $w$ is constant. For $t_0,t_1\in \Z$ we define the interval $[t_0, t_1] := \{t\in \Z: t_0\leq t\leq t_1\}$, and if $w\in (\R^q)^\Z$ is a map and $t_0\leq t_1$, then $w_{[t_0, t_1]}$ denotes the vector $\begin{bmatrix}w(t_0)^\top & \cdots & w(t_1)^\top\end{bmatrix}^\top$. For a set $\B\subseteq (\R^q)^\Z$, we define $\B_{[t_0, t_1]} := \left\{w_{[t_0, t_1]} : w\in \B \right\}$. If $X,Y$ are subsets of the same linear space, then $X+Y$ denotes their Minkowski sum, and when $X = \{x_0\}$ we write $x_0 + Y := \{x_0\}+Y$ for convenience. Finally, we denote the set of real polynomial matrices of size $p\times q$ by $\R^{p\times q}[\xi]$. 

\section{Behavioral preliminaries}
In the behavioral approach, a \textit{(dynamical) system} is not defined by a specific representation, but identified with a subset of a space of trajectories called its \textit{behavior}. In this paper, we focus on the particular case where $\B\subseteq (\R^q)^\Z$, that is, discrete-time systems taking values in $\R^q$. Moreover, we consider systems which are \textit{time-invariant}, that is, if $w\in \B$ then $\sigma^tw\in \B$ for all $t\in \Z$. Lastly, we make the (unrestrictive) assumptions that all behaviors under consideration are nonempty and \textit{complete}, i.e., there exists an $L\in \Z_+$ such that for every $w\in (\R^q)^\Z$ we have
\[
\text{$w_{[t, t+L]}\in \B_{[t, t+L]}$ for all $t\in \Z$} \ \implies\ w\in \B.
\]
If the integer $L$ is required in context, we say that the system is \textit{$L$-complete}. 

This paper focuses on the differences between linear and affine systems. For this, we will often need to refer to sets of behaviors. These are referred to as \textit{model classes}. We focus on two model classes of time-invariant and complete behaviors, namely
\begin{itemize}
	\item $\mathcal{L}^q$, the set of behaviors which are \textit{linear}, i.e.,
	\[
	w_1,w_2\in \B, \ \  \alpha,\beta\in \R \ \implies \ \alpha w_1 + \beta w_2\in \B;
	\]
	\item $\A^q$, the set of behaviors which are \textit{affine}, i.e.,
	\[
	w_1,w_2\in \B, \ \ \alpha\in \R \ \implies \ \alpha w_1 + (1-\alpha)w_2\in \B.
	\]
\end{itemize}
Observe that $\mathcal{L}^q\subsetneq \A^q$, as every linear behavior is affine, but the converse does not hold in general.

Like systems in $\mathcal{L}^q$, any system $\B\in\A^q$ can be represented in a number of ways (cf. \cite{Padoan2025}). Examples include affine state-space equations, representations in terms of measured data, or offset kernel representations, i.e.,
\[ \B = \ker_c R(\sigma):=\left\{w\in (\R^q)^\Z : R(\sigma)w = c\right\},\]
where $R\in \R^{p\times q}[\xi]$ and $c\in \R^p$. To make our results apply to each of these cases and avoid relying on a specific representation, we will work \textit{representation-free}, that is, without relying on a representation, whenever possible.

Although linear and affine systems are closely related, there are some important differences between them that make studying the stability of affine systems notably more involved. For instance, affine systems in state-space form need not admit an equilibrium point, whereas for linear systems the origin is always an equilibrium. To study the stability of affine systems, we will thus introduce a number of notions inspired by the analysis of nonlinear systems to the behavioral setting. We say that a system $\B\in \A^q$ is
\begin{itemize}
	\item \textit{autonomous} if for all $w_1,w_2\in \B$, $w_1(t) = w_2(t)$ for $t\leq 0$ implies $w_1 = w_2$;
	\item \textit{contractive} if $\lim_{t\to\infty}(w_1(t)-w_2(t)) = 0$ for all $w_1,w_2\in \B$;
	\item \textit{offset stable} if there exists a constant trajectory $\bar w\in \B$ and $\lim_{t\to\infty} w(t) = \bar w$ for all $w\in \B$.
\end{itemize}
Intuitively, a system is contractive if any two trajectories become indistinguishable, regardless of initial conditions.

Note that offset stable behaviors admit a \textit{unique} constant trajectory. We may therefore say that an offset stable behavior $\B\in \A^q$ with constant trajectory $\bar w\in \B$ is \textit{$\bar w$-stable}. In this notation, $0$-stability coincides with the definition of stability for linear behaviors used in e.g. \cite{Willems1998, Kojima2006}.

We will revisit the following, simplest nontrivial example illustrating the results throughout this paper. 
\begin{example}\label{ex:example}
	Consider the system $\B\subseteq\R^\Z$ consisting of all scalar trajectories $w:\Z\to\R$ such that for all $t\in\Z$,
	\[ 10w(t+2) + 18w(t+1)+9w(t) = 20. \]
	It is straightforward to see that $\B\in \A^1$ and that $\B$ is autonomous. By inspecting a number of trajectories of the system, we can see that each converges to the same constant trajectory. Indeed, rewriting the system centered at this fixed point, and deducing a linear state-space representation, we can verify offset stability using standard techniques. \hfill $\Box$
\end{example} 

However, the methodology of this example has a few downsides: it requires the existence and determination of the constant trajectory and it depends on a specific, non-unique, state realization. These are not immediate in more complex cases, illustrating the limitations of representation-dependent approaches. As such, we address the following problem:
\begin{enumerate}
	\item Given an autonomous system $\B\in\A^q$, provide representation-free necessary and sufficient tests to guarantee offset stability and contractivity. 
\end{enumerate}
Having examined autonomous affine systems, we then study control for \textit{non-autonomous} affine systems. Such a system $\B\in \A^{q+k}$ consists of trajectories which can be partitioned into $q$-dimensional \textit{to-be-controlled} variables and $k$-dimensional \textit{control variables}. In this setting, we investigate our last three problems:
\begin{enumerate}
	\setcounter{enumi}{1}
	\item Given a system $\B\in\A^{q+k}$ and an affine reference $\mathcal{R}\in\A^q$, provide tests for whether we can realize $\T$ by interconnecting $\B$ with an affine controller $\mathcal{C} \in \A^k$. 
	\item Given a system $\B\in\A^{q+k}$, provide necessary and sufficient conditions under which a contractive, implementable reference exists.
	\item Given a system $\B\in\A^{q+k}$, provide necessary and sufficient conditions under which the above can be achieved with a linear or affine controller. 
\end{enumerate}

We address these problems sequentially: the solution to problem 1) informs the tools needed for problems 3) and 4), which are in turn based on problem 2). 

\section{Stability theory for affine behaviors}\label{sec:stab}
In this section we study autonomous affine behaviors with the aim of resolving problem 1). A crucial insight enabling our analysis is the following result from \cite{Padoan2025}, which states that every affine system can be decomposed as the sum of a linear system and an \textit{offset} trajectory.

\begin{lemma}[{\cite[Theorem 3]{Padoan2025}}]\label{lem: lemma about difference behaviors of affine systems}
	Let $\B\in \A^q$. Then
	\begin{enumerate}[(i)]
		\item $\dif(\B)\in \mathcal{L}^q$, where $\dif(\B) := \B-\B$;
		\item $\B = w_0 + \dif(\B)$ for all $w_0\in \B$.
	\end{enumerate}
\end{lemma}

We refer to $\dif(\B)$ as the \textit{difference behavior (of $\B$)}. Our first main result characterizes contractivity of affine systems on the basis of the difference behavior. 

\begin{lemma}\label{lem: contractive iff difference behavior is stable}
	For $\B\in \A^q$, the following are equivalent: 
	\begin{enumerate}[(i)]
		\item $\B$ is contractive;
		\item $\B$ is offset stable;
		\item $\dif(\B)$ is $0$-stable.
	\end{enumerate}
\end{lemma}

The proof of Lemma~\ref{lem: contractive iff difference behavior is stable} hinges on proving that (i) implies (ii), that is, showing that for a contractive system all trajectories converge to a certain limit, and moreover that this limit must be a constant trajectory. The full details of the proof, as well as all other proofs for results in this paper, can be found in the Appendix. 

Lemma~\ref{lem: contractive iff difference behavior is stable} provides a number of avenues towards a test for these properties. Intuitively, testing for contractivity requires us to perform a test for all \textit{pairs of trajectories} of $\B$, whereas testing for offset stability requires us to find a constant trajectory and then perform a test for each trajectory. Lastly, item (iii) in Lemma~\ref{lem: contractive iff difference behavior is stable}  allows us to convert results about the $0$-stability of linear behaviors into results about the contractivity of affine behaviors through the difference behavior. Often, when models are derived directly from data, determining the difference behavior can be numerically unstable (cf. \cite{Markovsky2025}). However, given an offset kernel representation or a state-space realization, this yields:

\begin{corollary}[Consequences for representations]\label{cor: equivalent conditions for contractivity}
	Let $\B\in \A^q$. Then the following are equivalent:
	\begin{enumerate}[(i)]
		\item $\B$ is contractive;
		\item $\B$ admits an offset kernel representation $\B = \ker_c R(\sigma)$ where $c\in \R^q$ and $R\in \R^{q\times q}[\xi]$ is Schur, that is, the roots of $\det R$ lie in the unit disk;
		\item $\B$ admits a state-space representation\footnote{That is, $\B$ is the collection of all $y\in (\R^q)^\Z$ such that there is an $x\in (\R^n)^\Z$ such that
			$ x(t+1) = Ax(t) + E$ and $y(t) = Cx(t) + F$ for all $t\in \Z$, cf. \cite{Padoan2025}.} with matrices $(A,C, E, F)$ where $A\in \R^{n\times n}$ is Schur, that is, the eigenvalues of $A$ lie in the unit disk.
	\end{enumerate}
\end{corollary}

Note that every \textit{minimal} offset kernel or state-space representation of a contractive system is guaranteed to have the property described in (ii) and (iii) respectively.

\begin{example}[ctd.]
	Consider the system from Example~\ref{ex:example}. Given that $\B$ is already in the form of an offset kernel representation, we can prove that the system is contractive using Corollary~\ref{cor: equivalent conditions for contractivity}. Indeed, the $1\times 1$ polynomial matrix $R(\xi) =10\xi^2 + 18\xi+9$ is Schur and hence the system is contractive. By Lemma~\ref{lem: contractive iff difference behavior is stable}, it is therefore also offset stable. Moreover, the proof of Lemma~\ref{lem: contractive iff difference behavior is stable} elucidates how to find the unique constant trajectory contained in $\B$. Indeed, we see that $\bar w := \frac{20}{37}\in \B$, and that $\B$ is $\bar w$-stable. \hfill $\Box$
\end{example}

\subsection{Quadratic difference forms}
It is well-known that a linear system in state-space form is asymptotically stable if and only if it admits a quadratic Lyapunov function. In this subsection, we review the analogue of this result for linear behaviors, and then show how it generalizes to affine behaviors. To do so, we need the notion of a \textit{quadratic difference form (QDF)}. For convenience, we denote by $\mathbb{S}^k$ the set of real symmetric matrices of size $k$.

Let $\Phi\in \mathbb{S}^{q(L+1)}$ for some $L\in \Z_+$. The QDF (induced by $\Phi$) is defined as
\[
Q_{\Phi}(w)(t) := w_{[t, t+L]}^\top \Phi w_{[t, t+L]},
\]
where $w\in (\R^q)^\Z$. Note that the QDF $Q_\Phi$ is actually not uniquely determined by $\Phi$. By adding zero blocks to $\Phi$, we can obtain a larger matrix which leads to the same QDF. We thus define the \textit{degree} of $Q_{\Phi}$ as the smallest $\ell\in \Z_+$ such that there exists a $\bar\Phi\in \mathbb{S}^{q(\ell+1)}$ for which $Q_{\Phi} = Q_{\bar\Phi}$. This means that a QDF of degree $L$ can be represented by a matrix in $\mathbb{S}^{q(L+1)}$, and that if $\Phi\in \mathbb{S}^{q(L+1)}$ is given, then $Q_{\Phi}$ has degree at most $L$. Finally, we define the \textit{increment} of $Q_\Phi$ by
\[
\nabla Q_{\Phi}(w)(t) := Q_{\Phi}(w)(t+1) - Q_{\Phi}(w)(t),
\]
where, as usual, $w\in (\R^q)^\Z$. Note that $\nabla Q_{\Phi}$ is also a QDF: indeed, by taking
\[
\nabla \Phi := \begin{bmatrix}0 \\ I_{q(L+1)}\end{bmatrix} \Phi\begin{bmatrix}0 \\ I_{q(L+1)}\end{bmatrix}^\top -\begin{bmatrix} I_{q(L+1)} \\ 0\end{bmatrix} \Phi\begin{bmatrix} I_{q(L+1)} \\ 0\end{bmatrix}^\top
\]
we have $\nabla Q_{\Phi} = Q_{\nabla \Phi}$. From this we see that if $Q_{\Phi}$ has degree $L$, then $\nabla Q_{\Phi}$ has degree $L+1$.

The following result characterizes $0$-stability of a behavior $\B\in \mathcal{L}^q$ in terms of the existence of a QDF satisfying certain (non)negativity conditions.
\begin{theorem}[{\cite[Theorem 1]{Kojima2006}}]\label{thm: linear is stable iff there exist Lyapunov function}
	For $\B\in \mathcal{L}^q$, the following are equivalent:
	\begin{enumerate}[(i)]
		\item $\B$ is $0$-stable;
		\item There exists $L\in \Z_+$ and $\Phi\in \mathbb{S}^{q(L+1)}$ such that for all $w\in \B$ and $t\in \Z$,
		\[
		Q_{\Phi}(w)(t)\geq 0, \qquad \nabla Q_{\Phi}(w)(t)\leq 0,
		\]
		and $\nabla Q_{\Phi}(w) = 0$ if and only if $w = 0$.
	\end{enumerate} 
\end{theorem}

Given the clear links to Lyapunov functions for linear systems in state-space representation, we refer to a QDF $Q_{\Phi}$ which satisfies the conditions in (ii) as a \textit{Lyapunov function (for $\B$)}. In order to make it possible to test for this condition, we note that for $L$-complete systems, the degree of a Lyapunov function can be bounded by $L-1$ \cite[Proposition 4.9]{Willems1998}. This and Lemma~\ref{lem: contractive iff difference behavior is stable} lead to the following consequence for affine behaviors.

\begin{lemma}\label{cor: contractive iff there exists Lyapunov function for dif(B)}
	Let $\B\in \A^q$ be $L$-complete. Then the following are equivalent.
	\begin{enumerate}[(i)]
		\item  $\B$ is contractive;
		\item there exists a positive semidefinite $\Phi\in \mathbb{S}^{qL}$ such that for all $w_1,w_2\in\B$ and $t\in \Z$,
		\[\nabla Q_{\Phi}(w_1-w_2)(t)\leq 0 \quad \textrm{and}  \] 
		\[  \nabla Q_{\Phi}(w_1-w_2) = 0 \ \iff w_1=w_2;\]
		\item there exists $\Psi\in\mathbb{S}^{q(L+1)+1}$ such that $\Psi$ can be partitioned with $\Psi_{11} = \nabla \Phi$ for some positive semidefinite $\Phi\in \mathbb{S}^{qL}$ and
		\[    \begin{bmatrix} w_{[t, t+L]}\\ 1\end{bmatrix}^\top \Psi 
		\begin{bmatrix} w_{[t, t+L]}\\ 1\end{bmatrix} \leq 0,\]
		for all $w\in \B$ and $t\in\Z$, and there is precisely one trajectory $\bar{w}\in \B$ for which equality holds for all $t\in \Z$;
		\item there exists a positive semidefinite $\Phi\in \mathbb{S}^{qL}$ such that $Q_\Phi$ is a Lyapunov function for $\dif(\B)$.
	\end{enumerate} 
\end{lemma}

Analogously to the definition of Lyapunov functions, we refer to $Q_\Phi$ as in Corollary~\ref{cor: contractive iff there exists Lyapunov function for dif(B)}.(ii) as a \textit{contraction form}.

\begin{remark}[Computational issues]
	If we have access to a kernel representation of $\dif(\B)$, we can test for the condition (iii) using \textit{linear matrix inequalities (LMIs)} \cite{Kojima2006_LMI}. In the absence of a representation for $\dif(\B)$ and/or knowledge of the resulting equilibrium, testing for either of the other conditions can prove more convenient. In general, all conditions in Corollary~\ref{cor: contractive iff there exists Lyapunov function for dif(B)} are finite-dimensional and convex, allowing efficient solutions.
\end{remark}

\begin{example}[ctd.]
	Consider the system from Example~\ref{ex:example}. To test for contractivity, a standard approach is to attempt to find a \textit{contraction metric}, i.e., a norm $\norm{\cdot}$ and scalar $0\leq \gamma <1$ such that for all trajectories $w_1,w_2\in\B$ and all $t\in \Z$,
	\[ \norm{w_1(t+1)-w_2(t+1)} \leq \gamma \norm{w_1(t)-w_2(t)}. \]
	Even though the system $\B$ was shown to be contractive, it can be shown that for \textit{any} norm and $0\leq \gamma < 1$, there is a trajectory of $\B$ violating the above condition. In contrast, Corollary~\ref{cor: contractive iff there exists Lyapunov function for dif(B)} shows that there must exist a \textit{contraction form}. Routine calculation reveals that 
	\[ \Phi = \begin{bmatrix} 9 & 9 \\ 9 & 10 \end{bmatrix}\]
	indeed induces a contraction form. \hfill $\Box$
\end{example}
This demonstrates a fundamental advantage of the QDF-based approach: it can certify contraction in cases where no norm-based metric exists.

\section{Implementability}\label{sec:implementability}
Having established tools for analyzing autonomous affine behaviors, we now turn to the design question: given a non-autonomous system with control variables, can we shape its behavior to achieve desirable properties such as contraction? To investigate this in further detail, we first need to formalize the notions of `interconnection' and `control' in the behavioral framework. Let $q,k\in \Z_+$ and $\B\in \A^{q+k}$. The trajectories in $\B$ can be thought of as pairs $(w,c)$ where $w\in (\R^q)^\Z$ and $c\in (\R^k)^\Z$. In this interpretation, $w$ represents the \textit{to-be-controlled variables}, whereas $c$ represents the \textit{control variables}. For instance, if $\B$ is given in state-space form, $w$ could denote the state variables, whereas $c$ contains its input and output. 

Consider a system $\mathcal{C}\subseteq (\R^k)^\Z$. To interconnect $\B$ and $\mathcal{C}$, we define
\[
\B \Join \mathcal{C} := \{(w,c)\in \B : c\in \mathcal{C}\},
\]
which represents the subbehavior of $\B$ where the control variables are constrained by $\mathcal{C}$. Hence, we refer to $\mathcal{C}$ as a \textit{controller}. The \textit{interconnection of $\B$ and $\mathcal{C}$} is then defined as the projection of this subbehavior onto the to-be-controlled variables, i.e.
\[
\B\lVert \mathcal{C} := \{w\in (\R^q)^\Z : (w,c)\in \B\Join \mathcal{C}\},
\]
or equivalently, $\B\lVert \mathcal{C} = \pi_w(\B \Join \mathcal{C})$, where $\pi_w$ is the projection onto $(\R^q)^\Z$. In the case where $\mathcal{C} = (\R^k)^\Z$, the controller imposes no restrictions on the trajectories of $\B$, and we see that $\B\lVert \mathcal{C} = \pi_w(\B)$. Finally, when $\mathcal{C} = \{c_0\}$, we denote $\B\lVert c_0 := \B\lVert\{c_0\}$. 

In some particular cases, all variables of a system are available for control (e.g. when stabilizing a system via state feedback). This corresponds to the case where the trajectories of $\B$ have no to-be-controlled variables or, equivalently, $q = 0$. To avoid minor technicalities and ensure that our definitions above are consistent with this context, we identify $\B\Join \mathcal{C} = \B\lVert \mathcal{C} = \B\cap\mathcal{C}$ and $\pi_w(\B) = \B$ in this case. We refer to this special case as \textit{full interconnection}, and refer to the case $q>0$ as \textit{partial interconnection}.

Next, we consider the notion of \textit{implementability}. Let $\M$ be a model class (i.e., a collection of behaviors) and let $\B\in \A^{q+k}$. We say that a reference behavior $\T\in \A^q$ is \textit{$\M$-implementable} if there exists a $\mathcal{C}\in \M$ such that $\B\lVert\mathcal{C}=\T$. In this case, we say that \textit{$\mathcal{C}$ implements $\T$}. If useful for further clarification, we say that $\T$ is \textit{$(\M, \B)$-implementable}. Our next result is a direct generalization of \cite[Proposition 1]{Belur2001} to the affine case, and provides necessary and sufficient conditions for $\A^k$-implementability of an affine behavior.

\begin{theorem}[Affine implementability conditions]\label{thm: affine implementability theorem}
	Let $\B\in \A^{q+k}$ and $\T\in \A^q$. Fix any $w_0\in \T$. Then the following are equivalent:
	\begin{enumerate}[(i)]
		\item $\T$ is $\A^k$-implementable;
		\item $\T\subseteq \pi_w(\B)$ and for some $c_0\in (\R^k)^\Z$ such that $(w_0, c_0)\in \B$, we have $\B\lVert c_0\subseteq \T$. 
		\item $w_0+\dif(\B)\lVert 0 \subseteq \T \subseteq \pi_w(\B)$.
	\end{enumerate}
\end{theorem}
In fact, it can be shown that if condition (ii) holds for some $c_0\in (\R^k)^\Z$, it holds for any $c_0\in (\R^k)^\Z$ such that $(w_0, c_0)\in \B$.

\begin{remark}
	Using Theorem~\ref{thm: affine implementability theorem} to test for implementability involves checking twice whether all trajectories of an affine system are also trajectories of another affine system. For kernel representations this can be resolved using polynomial algebra. In the case where systems are given in a data-based representation, we can use linear algebra (cf. \cite[Corollary 1]{Padoan2023} for the linear case).
\end{remark}

\section{Existence of offset stable references}\label{sec:existence}
Often, we are not interested in implementing a specific reference behavior, but rather in achieving a more general control objective. For instance, given $\B\in \mathcal{L}^{q+k}$, we may want to find a controller $\mathcal{C}\in \mathcal{L}^k$ such that the resulting interconnection $\B\lVert \mathcal{C}$ is $0$-stable. In the language of implementability, this means we want to find a $0$-stable, $(\mathcal{L}^k, \B)$-implementable reference $\T\in \mathcal{L}^q$. In the full interconnection case, this particular problem is trivial. 

\begin{example}\label{exam: stabilization of linear behavior by non-regular interconnection}
	We consider a linear behavior in state-space form to illustrate this issue. Let $A\in \R^{n\times n}, B\in \R^{m\times n}$, and consider the behavior $\B\in \mathcal{L}^{n+m}$ consisting of all state-input pairs $(x,u)$ satisfying
	\[
	x(t+1) = Ax(t) + Bu(t)
	\]
	for all $t\in \Z$. In the classical problem of stabilization by state feedback, we have access to both the state and the input as control variables. Note that $\T = \{0\}$ is then a $0$-stable, $(\mathcal{L}^{n+m}, \B)$-implementable reference. Indeed, we can simply take $\mathcal{C} = \{0\}$ as the controller. \hfill $\Box$
\end{example}

This may seem rather strange at first glance, as it gives the impression that \textit{every} system $\B$ of the form given in Example~\ref{exam: stabilization of linear behavior by non-regular interconnection} is `stabilizable' in some sense. This is however not the case. To see why, we introduce the notion of \textit{regular} implementability.

\subsection{Regular implementability}

We first define what we mean by a \textit{regular} interconnection. The intuitive idea is that an interconnection between a system and controller is \textit{regular} if the controller does not over-constrain the to-be-controlled system by providing it with too many input signals. To make this more formal, we need to discuss the notion of inputs and outputs from a behavioral point of view.

Let $\B\in \A^q$ and let $m\in [0,q]$. Let $p := q-m$, and decompose $\R^q$ as $\R^m\times \R^p$. Let $\pi_u : \R^m\times \R^p\to\R^m$ be the projection map. We say that $(\R^m)^\Z$ is \textit{free in $\B$} if there exists a permutation matrix $\Pi\in \R^{q\times q}$ such that $\pi_u(\Pi \B) = (\R^m)^\Z$.\footnote{Here, we interpret $\Pi$ and $\pi_u$ as maps $(\R^q)^\Z\to(\R^q)^\Z$ and $(\R^m\times\R^p)^\Z\to(\R^m)^\Z$ respectively.} We denote
\begin{align*}
	\bm m(\B) &:= \max\{m\in [0,q] : \text{$(\R^m)^\Z$ is free in $\B$}\}, \\ 
	\bm p(\B) &:= q - \bm m(\B),
\end{align*}
and refer to $\bm m(\B)$ and $\bm p(\B)$ as the \textit{input} and \textit{output cardinality of $\B$} respectively. An important fact is that $\bm m(\B) = \bm m (\dif(\B))$ and $\bm p(\B) = \bm p(\dif(\B))$, see \cite[p. 13]{Padoan2025}. 

Now, given $\B\in \A^{q+k}$ and $\mathcal{C}\subseteq (\R^k)^\Z$, we say that the interconnection of $\B$ and $\mathcal{C}$ is \textit{regular} if $\B\Join \mathcal{C}$ is nonempty and
\[
\bm m(\B \Join \mathcal{C}) = \bm m(\B) - \bm p(\mathcal{C}),
\]
or put less formally, the number of inputs of the interconnected system has to equal the original number of inputs minus the number of outputs supplied by the controller. The terminology for regular implementability is analogous to that of implementability, interjecting the word `regular(ly)' where needed.

Notice that in Example~\ref{exam: stabilization of linear behavior by non-regular interconnection} we have $\bm p(\mathcal{C}) = n+m$, whereas $\bm m(\B) = m$. In less mathematical terms, the controller $\mathcal{C}$  supplies more outputs than the to-be-controlled system has inputs. To make the stabilization problem well-posed, we therefore restrict our attention to \textit{regularly} implementable references.

\begin{remark}
	Given $\B\in \A^{q+k}$ and $\mathcal{C}\in \A^k$, one can show that the interconnection of $\B$ and $\mathcal{C}$ is regular if and only if the interconnection of $\dif(\B)$ and $\dif(\mathcal{C})$ is regular. This allows for the use of the linear conditions (cf. \cite{Belur2002}) in testing the regularity of an affine interconnection.
\end{remark}

\subsection{Sufficiency of linear and affine controllers}
Another well-known result about linear systems in state-space form is that a system can be stabilized by linear dynamic output feedback if and only if it is detectable and stabilizable. There is an analogous result in the linear case for behaviors (cf. \cite[Theorem 6]{Belur2002}), which we now generalize to the affine setting, resolving problem 3). To do so, we first need a few definitions.

Let $\B\in \A^q$. Given a constant trajectory $\bar w\in (\R^q)^\Z$, we say that $\B$ is \textit{$\bar w$-stabilizable} if for every $w\in \B$, there is a $\tilde w\in \B$ such that $w(t) = \tilde w(t)$ for $t\leq 0$ and $\lim_{t\to\infty}\tilde w(t) = \bar w$. We say $\B$ is \textit{offset stabilizable} if there exists a constant trajectory $\bar w\in \B$ such that $\B$ is $\bar w$-stabilizable.

Similarly to Lemma~\ref{lem: contractive iff difference behavior is stable}, we can now  make an observation:
\begin{lemma}\label{lem: offset stabilizable iff dif. behavior stabilizable}
	Let $\B\in\A^q$. Then $\B$ is offset stabilizable if and only if $\dif(\B)$ is $0$-stabilizable. 
\end{lemma}
Finally, for $\B\in \A^{q+k}$, we say that $\B$ is \textit{detectable} if
\[
(w_1,c), (w_2,c)\in \B \ \ \implies \ \ \lim_{t\to\infty} (w_1(t)-w_2(t)) = 0.
\]
That is, if two trajectories of $\B$ have the same values in the control variables, then the trajectories contract. Detectability of an affine system turns out to be equivalent to the existence of an offset stable, implementable affine reference. 
\begin{lemma}[Equivalent conditions for detectability]\label{lem: equivalent conditions for detect. of affine behavior}
	For $\B\in \A^{q+k}$, the following are equivalent:
	\begin{enumerate}[(i)]
		\item $\B$ is detectable;
		\item $\dif(\B)$ is detectable;
		\item there exists an offset stable $\A^k$-implementable $\T\in \A^q$.
	\end{enumerate}
\end{lemma}
Note that this reference need not be \textit{regularly} implementable however. As illustrated in Example~\ref{exam: stabilization of linear behavior by non-regular interconnection}, this may yield trivial results. In order to obtain a contractive, \textit{regularly} implementable affine reference, we need an additional stabilizability condition. Before presenting the main result of this section, we provide representation-based tests for detectability and offset stabilizability. 

\begin{lemma}\label{cor: tests for detectability/contractibility}
	Let $\B\in \A^{q+k}$. The following are equivalent:
	\begin{enumerate}[(i)]
		\item $\B$ is detectable and $\pi_w(\B)$ is offset stabilizable;
		\item $\B$ admits an offset kernel representation of the form
		\[
		\B = \ker_{\footnotesize\begin{bmatrix}\eta_1 \\ \eta_2\end{bmatrix}} \begin{bmatrix}R_{11}(\sigma) & R_{12}(\sigma) \\ R_{21}(\sigma) & 0\end{bmatrix}
		\]
		where $\rank R_{12}(\lambda) = k$ and $\rank R_{21}(\lambda)$ is constant for all $\lambda\in \C$ with $|\lambda|\geq 1$;
		\item $\B$ admits a state-space representation with matrices $(A, B, C, D, E, F)$ where $(C, A)$ is detectable and $(A, B)$ is stabilizable.
	\end{enumerate}
\end{lemma}

With these conditions in place, we can state the main result of this section.

\begin{theorem}\label{thm: partial interconnection result about contraction by regular interconnection}
	For $\B\in \A^{q+k}$, the following are equivalent:
	\begin{enumerate}[(i)]
		\item $\B$ is detectable and $\pi_w(\B)$ is offset stabilizable;
		\item\label{item:lin cont}  there exists a regularly $\mathcal{L}^k$-implementable, offset stable reference $\T\in \A^q$.
	\end{enumerate}
	Moreover, if (i) holds and $(\bar w, \bar c)\in \B$ is a constant trajectory, then there exists a regularly $\A^k$-implementable, $\bar w$-stable $\T\in \A^q$.
\end{theorem}
An interesting detail of Theorem~\ref{thm: partial interconnection result about contraction by regular interconnection} can be found in (\ref{item:lin cont}): if $\B\in \A^{q+k}$ is detectable and $\pi_w(\B)$ is offset stabilizable, we can always make the system offset stable (or equivalently, contractive (cf. Lemma~\ref{lem: contractive iff difference behavior is stable})) using a \textit{linear} controller. On the other hand, using affine controllers allows us to choose the equilibrium to which resulting closed-loop system converges. 

\section{Conclusion}
In this paper, we investigated stability, contraction, stabilizability, and controllers for affine systems. To be precise, we investigated foundational, representation-free existence results. After defining contraction and offset stability for autonomous affine systems, we showed their equivalence, giving rise to tests in terms of the difference behavior. After this, we generalized the use of quadratic difference forms from their use in stability analysis of linear systems towards playing a role as quadratic contraction forms. Moving our attention to control, we provided implementability conditions for the affine case, and derived a characterization for the existence of a regularly implementable offset stable reference. Interestingly, for this linear controllers are sufficient, and we require affine controllers only to place the fixed point of the closed-loop system. Our results show that the behavioral framework is a natural home for stability analysis of affine systems, with the difference behavior serving as the bridge to linear theory.

\bibliography{references}
\bibliographystyle{IEEEtran}

\renewcommand\thesection{\Alph{section}}
\setcounter{section}{16}
\section*{Appendix: Proofs}

\noindent\textit{Proof of Lemma~\ref{lem: contractive iff difference behavior is stable}.} The equivalence of (i) and (iii) is clear by definition, and (ii) evidently implies (i). It remains to show (i) $\Rightarrow$ (ii). Assume that $\B\in \A^q$ is contractive. Then $\B = \ker_c R(\sigma)$ for some Schur $R\in \R^{q\times q}[\xi]$ and $c\in \R^q$, cf. Corollary~\ref{cor: equivalent conditions for contractivity}. Since $R$ is Schur, $\det R(1)\neq 0$. Let $\bar w := R(1)^{-1} c$. Then $R(\sigma)\bar w = c$, so $\bar w\in \B$. Moreover, $\B - \bar w = \dif(\B)$ (cf. Lemma~\ref{lem: lemma about difference behaviors of affine systems}), so $\B-\bar w$ is $0$-stable by the equivalence of (i) and (iii). This means $\B$ is $\bar w$-stable as desired. \qed \\

\noindent\textit{Proof of Corollary~\ref{cor: equivalent conditions for contractivity}.} 
We have
\begin{align*}
	\text{\small $\B$ is contractive} \ \ &\xLeftrightarrow{\text{Lemma~\ref{lem: contractive iff difference behavior is stable}}} \ \ \text{\small $\dif(\B)$ is $0$-stable} \\
	&\xLeftrightarrow{\text{\cite[Lemma 2]{Kojima2006}}}
	\begin{array}{c}
		\text{\small $\dif(\B) = \ker R(\sigma)$ for some} \\
		\text{\small Schur $R\in \R^{q\times q}[\xi]$}
	\end{array} \\
	&\xLeftrightarrow{\text{\cite[Theorem 4]{Padoan2025}}}
	\begin{array}{c}
		\text{\small $\B = \ker_c R(\sigma)$ for some} \\
		\text{\small Schur $R\in \R^{q\times q}[\xi]$, $c\in \R^q$},
	\end{array}
\end{align*}
which shows (i) $\Leftrightarrow$ (ii). The equivalence (i) $\Leftrightarrow$ (iii) follows similarly:
\begin{align*}
	\text{\small $\B$ is contractive} \ \ &\xLeftrightarrow{\text{Lemma~\ref{lem: contractive iff difference behavior is stable}}} \ \ \text{\small $\dif(\B)$ is $0$-stable} \\
	&\xLeftrightarrow{\text{\hspace{1cm}}}
	\begin{array}{c}
		\text{\footnotesize $\dif(\B)$ admits a state-space representation} \\
		\text{\footnotesize $(A, C)$ with $A$ Schur}
	\end{array} \\
	&\xLeftrightarrow{\text{\cite[Theorem 6]{Padoan2025}}}
	\begin{array}{c}
		\text{\footnotesize $\B$ admits a state-space representation} \\
		\text{\footnotesize $(A,C,E,F)$ with $A$ Schur},
	\end{array}
\end{align*}
where the second equivalence holds since we can take a minimal state-space representation of $\dif(\B)$ such that $A$ is Schur, cf. \cite[Lemma 3.1]{Valcher2002}. \qed \\

\noindent\textit{Proof of Corollary~\ref{cor: contractive iff there exists Lyapunov function for dif(B)}.} We begin the proof with two observations. The equivalence (i) $\Leftrightarrow$ (iv) follows from Theorem~\ref{thm: linear is stable iff there exist Lyapunov function} and Lemma~\ref{lem: contractive iff difference behavior is stable}. Moreover, (ii) and (iv) are equivalent by definition. We conclude the proof by showing two additional implications. 

(ii) $\Rightarrow$ (iii). Let $\Phi$ be such that (ii) holds. By the previous, we know that if (ii) holds, then $\B$ is contractive, and by Lemma~\ref{lem: contractive iff difference behavior is stable}, we can conclude that there exists a unique constant trajectory $\bar{w}\in\B$. Let $w\in\B$, and apply the conditions of (ii) with $w_1=w$ and $w_2=\bar{w}$. Note that $\nabla Q_{\Phi}(w-\bar{w})(t)$ can be written as 
\[ \begin{bmatrix} w_{[t, t+L]}\\ 1\end{bmatrix}^\top {\footnotesize\begin{bmatrix} \nabla \Phi & -\bar{w}_{[t,t+L]}^\top \nabla \Phi \\ -\nabla \Phi \bar{w}_{[t,t+L]} & \bar{w}_{[t,t+L]}^\top \nabla \Phi\bar{w}_{[t,t+L]} \end{bmatrix}} 
\begin{bmatrix} w_{[t, t+L]}\\1 \end{bmatrix} \] 
By defining the middle matrix as $\Psi$, the result follows. 

(iii)$\Rightarrow$(iv):
Let $\tilde{w}\in \dif(\B)$. Then $-\tilde w\in \dif(\B)$ as well, so $w_\pm := \bar w \pm \tilde w\in \B$. Applying the given condition on $\Psi$ to $w_+$ and $w_-$, a straightforward computation yields
\[
\nabla Q_{\Phi}(\pm \tilde w)(t) \leq \pm 2 \bar w_{[t, t+L]}^\top \nabla \Phi \tilde w_{[t, t+L]} \pm 2\tilde w^\top_{[t, t+L]} \Psi_{12}\]
for all $t\in \Z$.
Observing that $\nabla Q_{\Phi}(\tilde w) = \nabla Q_{\Phi}(-\tilde w)$, it follows by adding these two inequalities that $\nabla Q_{\Phi}(\tilde w)(t)\leq 0$ for all $t\in \Z$. Moreover, we see that if $\nabla Q_{\Phi}(\tilde w) = 0$, then 
\[
\footnotesize\begin{bmatrix} (w_+)_{[t, t+L]} \\ 1\end{bmatrix}^\top \Psi\begin{bmatrix} (w_+)_{[t, t+L]} \\ 1\end{bmatrix} = 0 
\]
for all $t\in \Z$. Therefore $w_+ = \bar w$, i.e., $\tilde w = 0$. This shows that $Q_{\Phi}$ is a Lyapunov function for $\dif(\B)$. \qed \\
	
	\noindent\textit{Proof of Theorem~\ref{thm: affine implementability theorem}.} We show that (i) and (iii) are equivalent, and that (ii) and (iii) are equivalent.
	
	(i) 9$\Rightarrow$ (iii). Suppose $\mathcal{C}\in \A^k$ implements $\T$. It is clear that $\T\subseteq \pi_w(\B)$. Let $\tilde w\in \dif(\B)\lVert 0$. Then there exist $(w_i,c_i)\in \B$, $i\in \{1,2\}$, such that $(\tilde w,0) = (w_1 - w_2, c_1-c_2)$. Since $w_0\in \T=\B\lVert \mathcal{C}$, we have $(w_0, c_0)\in \B$ for some $c_0\in \mathcal{C}$. Hence
	\[
	(w_0+\tilde w,c_0) = (w_0,c_0) + (w_1,c_1) - (w_2,c_2)\in \B
	\]
	because $\B$ is affine. Therefore $w_0 + \tilde w\in \T$, which shows that (iii) holds.
	
	(iii) $\Rightarrow$ (i). Assume that $\T\in \A^q$ satisfies (iii). Let $\B = \ker_\eta R(\sigma)$ with $R\in \R^{p\times (q+k)}[\xi], \eta\in \R^p$ be an offset kernel representation of $\B$, and partition $R$ as $R = \begin{bmatrix}R_1 & R_2\end{bmatrix}$ with $R_1\in \R^{p\times q}[\xi]$ and $R_2\in \R^{p\times k}[\xi]$. Then $\dif(\B) = \ker R(\sigma)$ (cf. \cite[Theorem 4]{Padoan2025}), and therefore $\dif(\B)\lVert0 = \ker R_1(\sigma)$. Since $\dif(\B)\lVert 0\subseteq \dif(\T)$ by (iii), it follows from \cite[Lemma 14]{Willems2002} that $\dif(\T) = \ker (XR_1)(\sigma)$ for some $X\in \R^{\ell\times p}[\xi]$. Therefore $\T = \ker_{\zeta}(XR_1)(\sigma)$ for some $\zeta\in \R^\ell$, see \cite[Theorem 4]{Padoan2025}. Now let
	\[
	\mathcal{C} := \ker_{X(1)\eta - \zeta} (XR_2)(\sigma)\in \A^k.
	\]
	It is then easily verified that indeed $\T = \B\lVert \mathcal{C}$, proving that $\T$ is $\A^k$-implementable. 
	
	(ii) $\Rightarrow$ (iii). Let $\tilde w\in \dif(\B)\lVert 0$. Then $(\tilde w,0) = (w_1,c_1) - (w_2,c_2)$ for some $(w_1,c_1),(w_2,c_2)\in \B$. Hence
	\[
	(w_0+\tilde w,c_0) = (w_0,c_0) + (w_1,c_1)-(w_2,c_2)\in \B
	\]
	since $\B$ is affine. In particular, $w_0 + \tilde w\in \B\lVert c_0 \subseteq \T$ as required.
	
	(iii) $\Rightarrow$ (ii). Let $w\in \B\lVert c_0$. Then $(w,c_0)\in \B$. This means that $(w-w_0,0)\in \dif(\B)$, so $w-w_0\in \dif(\B)\lVert 0$. Hence $w\in w_0 + \dif(\B)\lVert 0\subseteq \T$. \qed \\
	
	\noindent\textit{Proof of Lemma~\ref{lem: offset stabilizable iff dif. behavior stabilizable}.} It is straightforward to verify that offset stabilizability of $\B$ implies $0$-stabilizability of $\dif(\B)$. Conversely, assume that $\dif(\B)$ is $0$-stabilizable. Let $\B = \ker_c R(\sigma)$ be a minimal offset kernel representation with $R\in \R^{\bm p(\B)\times q}[\xi]$, $c\in \R^{\bm p(\B)}$. Then $\dif(\B) = \ker R(\sigma)$, and since $\dif(\B)$ is $0$-stabilizable, $\rank R(\lambda) = \bm p(\B)$ for all $\lambda\in \C$ with $|\lambda|\geq 1$, see \cite{Eising2025}. Hence $R(1)$ has full row rank. Let $R(1)^\dagger$ be any right inverse of $R(1)$, and let $\bar w := R(1)^\dagger c$. Then $R(\sigma)\bar w = c$, so $\bar w\in \B$. Therefore $\dif(\B) = \B-\bar w$ (cf. Lemma~\ref{lem: lemma about difference behaviors of affine systems}), from which it is easy to see that $\B$ is $\bar w$-stabilizable as desired. \qed \\
	
	\noindent\textit{Proof of Lemma~\ref{lem: equivalent conditions for detect. of affine behavior}.} We prove (i) $\Rightarrow$ (ii) $\Rightarrow$ (iii) $\Rightarrow$ (i).
	
	(i) $\Rightarrow$ (ii). Let $(\bar w_1, \bar c), (\bar w_2,\bar c)\in \dif(\B)$. Then there exist $w_i\in (\R^q)^\Z$ and $c_i\in (\R^k)^\Z$ such that
	\[
	(\bar w_1, \bar c) = (w_1 - w_2, c_1-c_2), \ \ (\bar w_2, \bar c) = (w_3 - w_4, c_3 - c_4),
	\]
	and $(w_i, c_i)\in \B$, $i\in \{1,2,3,4\}$. By the equalities above, we have $c_1 + c_4 = c_2 + c_3 =: c$. Since $\B$ is affine, it follows that the midpoints $\frac12(w_1+w_4, c)$, $\frac12(w_2+w_3, c)$ belong to $\B$. As $\B$ is detectable, this yields 
	\begin{align*}
		0 &= \lim_{t\to\infty} (w_1(t) + w_4(t) - (w_2(t) + w_3(t))) \\
		&= \lim_{t\to\infty} (\bar w_1(t) - \bar w_2(t)),
	\end{align*}
	which proves that $\dif(\B)$ is detectable.
	
	(ii) $\Rightarrow$ (iii). Let $(w_0, c_0)\in \B$, and choose $\T := \B\lVert c_0$. Then $\T\in \A^q$, and $\T$ is clearly $\A^k$-implementable. Now let $w_1,w_2\in \T$. Then $(w_1,c_0),(w_2,c_0)\in \B$, so $w_1-w_2\in \dif(\B)\lVert 0$. Because $0\in\dif(\B)$, it follows that $\lim_{t\to\infty} (w_1(t) - w_2(t)) = 0$, which shows that $\T$ is contractive. Hence $\T$ is offset stable by Lemma~\ref{lem: contractive iff difference behavior is stable}.
	
	(iii) $\Rightarrow$ (i). Let $(w_1, c), (w_2,c)\in \B$, and suppose that $\T\in \A^q$ is offset stable and $\A^k$-implementable. Pick any $w_0\in \T$, and observe that $w_1-w_2\in \dif(\B)\lVert 0$. Hence $w_0 + w_1 - w_2 \in \T$ by Theorem~\ref{thm: affine implementability theorem}. Therefore
	\[
	\lim_{t\to\infty} (w_0(t) + w_1(t) - w_2(t)) - w_0(t))  = 0
	\]
	since $\T$ is also contractive (cf. Lemma~\ref{lem: contractive iff difference behavior is stable}). Hence $\B$ is detectable and the proof is complete. \qed \\
	
	\noindent\textit{Proof of Corollary~\ref{cor: tests for detectability/contractibility}.} (i) $\Leftrightarrow$ (ii). Let $\ker \begin{bmatrix}R_1(\sigma) & R_2(\sigma)\end{bmatrix}$ be a minimal kernel representation of $\dif(\B)$ as in the proof of Theorem~\ref{thm: partial interconnection result about contraction by regular interconnection}. By pre-multiplying $R_2$ with a suitable unimodular matrix $U\in \R^{\bm p(\B)\times \bm p(\B)}[\xi]$, we obtain
	\[
	\footnotesize U(\xi)R_2(\xi) = \begin{bmatrix}R_{12}(\xi) \\ 0\end{bmatrix},
	\]
	where $R_{12}$ has full row rank over $\R[\xi]$. Partition $R_1$ conformably. Then $\dif(\B)$ is represented minimally by
	\[
	\footnotesize\dif(\B) = \ker\begin{bmatrix}R_{11}(\sigma) & R_{12}(\sigma) \\ R_{21}(\sigma) & 0\end{bmatrix}.
	\]
	By \cite[Lemma P.2]{Willems1991}, the map $c\mapsto R_{12}(\sigma)c$ is surjective. Hence $\pi_w(\dif(\B)) = \ker R_{21}(\sigma)$. We thus see that $\dif(\B)$ is detectable if and only if $\rank R_{12}(\lambda) = k$ for all $\lambda\in \C$ with $|\lambda|\geq 1$ (cf. \cite{Polderman1998}), and that $\pi_w(\dif(\B))$ is $0$-stabilizable if and only if $\rank R_{21}(\lambda)$ is constant for $\lambda\in \C$ with $|\lambda|\geq 1$ (see \cite{Eising2025}). The desired result now follows from Lemma~\ref{lem: equivalent conditions for detect. of affine behavior} and Lemma~\ref{lem: offset stabilizable iff dif. behavior stabilizable}.
	
	(i) $\Leftrightarrow$ (iii). Let $(A,B,C,D)$ be the matrices of a minimal state-space representation of $\dif(\B)$. Following the same reasoning as in \cite[Section V]{Belur2002}, we see that $\dif(\B)$ is detectable if and only if $(C, A)$ is detectable, and $\pi_w(\dif(\B))$ is $0$-stabilizable if and only if $(A,B)$ is stabilizable. Using Lemma~\ref{lem: equivalent conditions for detect. of affine behavior} and Lemma~\ref{lem: offset stabilizable iff dif. behavior stabilizable}, the result follows. \qed \\
	
	\noindent\textit{Proof of Theorem~\ref{thm: partial interconnection result about contraction by regular interconnection}.} 
	(ii) $\Rightarrow$ (i). Assume that $\T\in \A^q$ is offset stable and regularly $\A^k$-implementable. Let $\mathcal{C}\in \A^k$ be a controller which regularly implements $\T$. Observe that
	\begin{multline}\label{eq: computation of difference behavior of interconnection}
		\dif(\T) = \dif(\pi_w(\B\Join \mathcal{C})) = \pi_w(\dif(\B\Join \mathcal{C})) \\
		= \pi_w(\dif(\B)\Join \dif(\mathcal{C})) = \dif(\B)\lVert \dif(\mathcal{C}),
	\end{multline}
	where we have used the linearity of the map $\pi_w$ and the fact that $\dif(A_1\cap A_2) = \dif(A_1)\cap \dif(A_2)$ if $A_1,A_2$ are affine sets and $A_1\cap A_2\neq \emptyset$. Moreover,
	\begin{multline}\label{eq: computation with i/o-cardinalities}
		\bm m(\dif(\B)\Join \dif(\mathcal{C})) = \bm m(\dif(\B\Join \mathcal{C})) = \bm m(\B\Join \mathcal{C}) \\
		= \bm m(\B) - \bm p(\mathcal{C}) = \bm m(\dif(\B)) - \bm p(\dif(\mathcal{C})),
	\end{multline}
	so we conclude that $\dif(\T)$ is $0$-stable (cf. Lemma~\ref{lem: contractive iff difference behavior is stable}) and regularly $(\mathcal{L}^k, \dif(\B))$-implementable. Hence, by \cite[Theorem 6]{Belur2002}, it follows that $\dif(\B)$ is detectable and $\pi_w(\dif(\B)) = \dif(\pi_w(\B))$ is $0$-stabilizable. That is, $\B$ is detectable and $\pi_w(\B)$ is offset stabilizable (by Lemma~|ref{lem: equivalent conditions for detect. of affine behavior} and Lemma~\ref{lem: offset stabilizable iff dif. behavior stabilizable} respectively). 
	
	(i) $\Rightarrow$ (ii). Assume that $\B$ is detectable and that $\pi_w(\B)$ is offset stabilizable. Then $\dif(\B)$ is detectable and $\dif(\pi_w(\B)) = \pi_w(\dif(\B))$ is $0$-stabilizable (cf. Lemma~\ref{lem: equivalent conditions for detect. of affine behavior} and Lemma~\ref{lem: offset stabilizable iff dif. behavior stabilizable}). By \cite[Theorem 6]{Belur2002}, there thus exists an $\bar\T\in \mathcal{L}^q$ which is $0$-stable and regularly $(\mathcal{L}^k,\dif(\B))$-implementable. Let $\bar{\mathcal{C}}\in \mathcal{L}^k$ be a controller which regularly implements $\bar\T$, and let  $\dif(\B) = \ker R(\sigma)$, $\bar{\mathcal{C}} = \ker C(\sigma)$ be minimal kernel representations, where $R\in \R^{\bm p(\B)\times (q+k)}[\xi]$ and $C\in \R^{\bm p(\bar{\mathcal{C}})\times k}[\xi]$. 
	Partition $R$ as $\begin{bmatrix}R_1 & R_2\end{bmatrix}$, where $R_1\in \R^{\bm p(\B)\times q}[\xi]$ and $R_2\in \R^{\bm p(\B)\times k}[\xi]$. We have
	\[
	\footnotesize\dif(\B)\Join \bar{\mathcal{C}} = \ker\begin{bmatrix}R_1(\sigma) & R_2(\sigma) \\ 0 & C(\sigma)\end{bmatrix},
	\]
	and since the interconnection of $\dif(\B)$ and $\bar{\mathcal{C}}$ is regular, it holds that $\bm m(\dif(\B)\Join \bar{\mathcal{C}}) = \bm m(\dif(\B)) - \bm p(\bar{\mathcal{C}})$, or equivalently, $\bm p(\dif(\B)\Join \bar{\mathcal{C}}) = \bm p(\dif(\B)) + \bm p(\bar{\mathcal{C}})$. This clearly implies that the polynomial matrix
	\[
	{\footnotesize\begin{bmatrix}R_1 & R_2 \\ 0 & C\end{bmatrix}}
	\]
	has full row rank over the ring $\R[\xi]$. Therefore the map
	\[
	\footnotesize\begin{bmatrix}w \\ c\end{bmatrix}\mapsto \begin{bmatrix} R_1(\sigma) & R_2(\sigma) \\ 0 & C(\sigma)\end{bmatrix}\begin{bmatrix}w \\ c\end{bmatrix}
	\]
	is surjective (cf. \cite[Lemma P.2]{Willems1991}). Pick $\eta\in \R^{\bm p(\B)}$ such that $\B = \ker_\eta R(\sigma)$, and choose any $\zeta\in \R^{\bm p(\B)}$. Let $\mathcal{C}_\zeta := \ker_{\zeta} C(\sigma)$. Then the set
	\[
	{\footnotesize\B \Join \mathcal{C}_\zeta = \ker_{\begin{bmatrix}\eta \\ \zeta\end{bmatrix}}\begin{bmatrix} R_1(\sigma) & R_2(\sigma) \\ 0 & C(\sigma)\end{bmatrix}},
	\]
	is nonempty. Hence $\dif(\B\lVert \mathcal{C}_\zeta) = \bar \T$ since $\dif(\mathcal{C}) = \bar{\mathcal{C}}$ (cf. \eqref{eq: computation of difference behavior of interconnection}), meaning $\B\lVert \mathcal{C}_\zeta$ is contractive (and therefore offset stable, cf. Lemma~\ref{lem: contractive iff difference behavior is stable}) and the interconnection of $\B$ and $\mathcal{C}_\zeta$ is regular (cf. \eqref{eq: computation with i/o-cardinalities}) for any choice of $\zeta\in \R^{\bm p(\bar{\mathcal{C}})}$. Choosing $\zeta = 0$, (ii) follows since $\mathcal{C}_0\in \mathcal{L}^k$.
	
	If in addition $(\bar w, \bar c)\in \B$ is a constant trajectory, then by choosing $\zeta = C(1)\bar c$ we see that $\B\lVert \mathcal{C}_\zeta$ is $\bar w$-stable. \qed \\

\end{document}